\documentclass[reqno,12pt]{amsart}

\usepackage{amssymb,latexsym}
\usepackage{fullpage}

\theoremstyle{plain}
\newtheorem{theorem}{Theorem}[section]
\newtheorem{corollary}{Corollary}[section]
\newtheorem{lemma}{Lemma}[section]

\newtheorem{proposition}{Proposition}[section]
\theoremstyle{remark}

\newtheorem{remark}{Remark}[section]
\newtheorem{example}{Example}[section]

\numberwithin{equation}{section}

\newcommand{\bC}{{\mathbb C}}
\newcommand{\bZ}{{\mathbb Z}}

\newcommand{\tr}{\operatorname{tr}}
\newcommand{\Pf}{\operatorname{Pf}}

\newcommand{\of}{{\mathfrak o}_{2n} }
\newcommand{\off}{\hat{{\mathfrak o}}_{2n} }

\newcommand{\fg}{\mathfrak g }
\newcommand{\fh}{\mathfrak  h}

\title{ A new  formula for  Pfaffian-type Segal-Sugawara  vector}
\author{Natasha Rozhkovskaya}
\address{ Kansas State University  Department of Mathematics, 138 Cardwell Hall,  Manhattan, KS, 66502}
\begin{document}
\maketitle
\begin{abstract}
A  combinatorial  formula  for Pfaffian  for the universal  enveloping algebra  $U(\off)$ of the affine Kac--Moody algebra $\off$ is proved.    It  allows  easily  to compute  the image of this  Segal-Sugawara vector   under  the Harish-Chandra homomorphism  and  to deduce  formulas   for classical  Pfaffian of   universal enveloping algebra $U(\of)$ of the  even orthogonal Lie algebra.
\end{abstract}
\section{Introduction}

Let $U(\fg)$ be  the universal enveloping algebra of a simple Lie algebra $\fg$. The Harish-Chandra isomorphism  identifies the  center $Z(\fg)$ of the algebra $U(\fg)$  with  the    algebra  of polynomials over the  Cartan subalgebra $\fh$ of $\fg$ that are  invariant under a certain  action of the corresponding Weyl group $W$.  The elements of  $Z(\fg)$  act in finite-dimensional irreducible representations of $\fg$ by multiplication by scalars, and these eigenvalues  can be found  from the images of central elements under the  Harish-Chandra isomorphism.
 
For  the even orthogonal Lie algebra $\fg=\of$  the center is generated by  $n$  elements  those images  under the  Harish-Chandra isomorphism are  $W$-invariant polynomials over $\fh$ of degrees $2,4,6,\dots,  2n-2$ and $n$.
There are several  constructions  that define these generators explicitly (see, e.g. \cite{HU, MN, PP, AW}).  In particular,  the generator of degree $n$ can be realized as a non-commutative  Pfaffian of a certain matrix with coefficients in $U(\of)$.
Namely,  in the notations of Section  \ref{ext}, if $F$ is the $2n\times 2n$ matrix  those  entries $\{F_{ij}\}$, $(i,j\in \{-n,\dots, -1,1,\dots n\})$ are  standard generators of $U(\of)$, and $J_{2n}$  is the  matrix  of the form that defines  $\of$, the Pfaffian of $\tilde F= FJ_{2n}$  is the central element defined by the formula 
\begin{align}\label{defpf}
\text{Pf}\,(\tilde F)=\frac{1}{2^n n!}\sum_{\sigma\in S_{[-n;n]} }\text{sgn}\,{(\sigma)} (FJ_{2n})_{\sigma(-n)\sigma(-n+1)}\dots  (FJ_{2n})_{\sigma(n-1)\sigma(n)}.
\end{align}
In \cite {TH}   a combinatorial  formula for $\text{Pf}\,\tilde F$  is proved,   which  (with some change of notations  to match the notations of  this note) can be formulated as follows:
\begin{align}\label{HP}
\text{Pf}\,(\tilde F)=\sum_{k=0}^{{\lfloor{n/2}\rfloor}} 
\sum_{\substack {I,J\subset \{-n,\dots, -1 \}\\|I|=|J|=2k}} \text{sgn}({\bar I, I}) \text{sgn} ({\bar J, J})\text{det}\, (\tilde F_{\bar I \bar J} +\rho(|J|)) \text{Pf}\, \tilde F_{-J\,J}\text{Pf}\, \tilde F_{I\,-I}.
\end{align}
 Here  for a subset $I$ of $\{-n,\dots, -1\}$,  the set  $\bar I$ is the complement    of $I$ in $\{-n,\dots, -1 \}$;  a  submatrix $\tilde F_{IJ}$  is  defined by
$\tilde F_{IJ}=(\tilde F_{ij})_{i\in I, j\in J}$;    $\rho(k)$ stands for the diagonal matrix $\rho(k)=\text{diag}\,(k-1,k-2,\dots, 1,0)$; the expressions 
 $\text{det}\, (\tilde F_{\bar I \bar J} +\rho(|J|))$, $\text{Pf}\, \tilde F_{-J\, J}$,  $\text{Pf}\, \tilde F_{I\, -I}$ are certain  noncommutative analogues of 
 the determinant and  Pfaffian of a matrix;  
  $\text{sgn}({\bar I, I})$,  $\text{sgn}({\bar J,J})$  are the signs of  certain permutations.  
(We refer the reader for full details to {\cite{TH}}.)
The advantage of this formula is that it immediately provides  the eigenvalues of the central element $\text{Pf}\, \tilde F$ on the highest weight modules. 

The purpose of this note is  to prove a  similar to (\ref{HP})  combinatorial  formula  for a Pfaffian of the matrix of generators  for the universal  enveloping algebra  $U(\off)$ of the affine Kac--Moody algebra $\off$.
using the techniques of \cite{IU} and \cite{TH}. The main results of this  paper are  Theorems \ref{thm1} and \ref{thm2}.

 In general, for a simple Lie algebra $\fg$ over $\bC$,  the  vacuum module $V_{-h^\vee}(\fg)$ over the affine Kac--Moody algebra $\hat \fg$  at the critical level  ${-h^\vee}$ is defined 
 as the quotient of the universal enveloping algebra  $U(\hat \fg)$ by the left ideal generated by $\fg[t]$ and $K+ h^\vee$, where $ h^\vee$ denotes the  
 dual Coxeter number  for $\fg$. As a vertex algebra,  $V_{-h^\vee}(\fg)$ has a non-trivial center $\mathfrak{z}(\hat \fg)$ defined by
  $$\mathfrak{z}(\hat \fg)=\{ S\in V_{-h^\vee}(\fg) \,|\, \fg[t]S=0\}.$$
  It structure is described  by the Feigin-Frenkel theorem (1992); see
\cite{F}   for detailed exposition.
    In \cite{MC, CT, M1}  complete sets of  explicit generators of the commutative algebra
    $\mathfrak{z}(\hat \fg)$ are constructed for  classical Lie algebras.  In the case of $\fg=\of$ the set of generators contains an element $\text{Pf\,}\tilde F[-1]$
defined below by  (\ref{defpf}), and this element  is the main focus of this note. 

\smallskip

Our goal is to prove a combinatorial formula  (Theorem \ref{thm1} and Theorem \ref{thm2})  for   $\text{Pf}\,\tilde F[-1]$   in the  spirit of formulas for Pfaffians 
proved in \cite{TH}.  The immediate application of this result  is that it  allows  to compute  easily the image of the  Pfaffian $\text{Pf}\,\tilde F[-1]$ under  the Harish-Chandra homomorphism (Section \ref{sec_HC}). 
The algebra  $\mathfrak{z}( \off)$ is  a  commutative subalgebra of  $U(t^{-1}\of[t^{-1}])^\fh$ --  the algebra of invariants under the action of the Cartan subalgebra $\fh$. A restriction of a  Harish-Chandra homomorphism 
 $U(t^{-1}\of[t^{-1}])^\fh\to U(t^{-1}\fh[t^{-1}]) $ to the subalgebra $\mathfrak{z}(\off)$ yields  an isomorphism of
 $\mathfrak{z}( \off)$ to the classical $\mathcal{W}$-algebra $\mathcal{W}(\of)$ (see \cite{F,MM}). The Harish-Chandra images of generators 
 of $\mathfrak{z}(\hat \fg)$   for classical Lie algebras $\fg$ were computed  in \cite{MC, MM}.  Below  Corollary {\ref{HC}}   deduces  the Harish-Chandra  image of $\text{Pf\,}\tilde F[-1]$  from the Theorem  \ref{thm1}  (or  Theorem \ref{thm2}) of the paper. 
 
 Moreover, we illustrate in Section \ref{fol} that the formulas of \cite{TH}  for Pfaffians in classical case can be deduced from the formulas  for affine Lie algebra $\off$ proved  in this note.

\subsection*{Acknowledgements} The author is very grateful to A.\,I.\,Molev for brining this problem to her attention and for  valuable discussions. She also would like to thank 
Max Planck Institute for Mathematics   for  warm hospitality.

\section{ Application of  exterior calculus }\label{ext}
Let $J_{2n}$ be a  $2n\times 2n$ symmetric matrix  with  ones  on the anti-diagonal and  zeros elsewhere.  Consider
 the corresponding  realization  of   $\of$ as a Lie algebra  of $2n\times 2n$ complex-valued  matrices $X$ that satisfy 
$ X^TJ_{2n}+J_{2n}X=0$.   
We will use the notation  $[-n;n] =\{-n,\dots, -1,1,\dots, n\}$ (the interval of integers  from $-n$ to $n$  with  the element $0$  being omitted). 
Let 
$\{E_{i,j}\}_{i,j\in [-n; n]}$  be the matrices that   have one in the $i$th row and $j$th column and zeros elsewhere. 
Set for $i,j\in{[-n;n]}$
$$
F_{i,j}=E_{i,j}- E_{-j,-i},\quad  ( F_{i,j}=-F_{-j,-i}).
$$
 The   Cartan  subalgebra $\fh$ of  $\of$ is generated by elements $F_{ii}$ ($i=-n,\dots, -1)$. 
Using the  bilinear form 
$$
\langle X,Y \rangle = \frac{1}{2} \tr (XY) \quad \text{for}~ X,Y\in \of,
$$
one defines the extended affine Kac--Moody algebra
$\off \oplus \bC\tau= \of[t, t^{-1}]\oplus \bC K\oplus \bC\tau$ and set
$X[r]=Xt^r$, $r\in \bZ$. Then the commutation relations in $\off$ are
\begin{align*}
[X[r], Y[s]]&=[X,Y][r+s] +r\delta_{r, -s} \langle X,Y \rangle K,\\
[\tau,X[r]]&=-r X[r-1],\quad [\off,K]=0.
\end{align*}
In particular,  for  $i,j,k,l\in[-n; n]$ one gets
\begin{align*}
[F_{i,j}[r],F_{k,l}[s] ]&= \delta_{j,k}F_{i,l}[r+s]+ \delta_{i,l}F_{-j,-k}[r+s]- \delta_{j,-l}F_{i,-k}[r+s]- \delta_{i,-k}F_{-j,l}[r+s]\\
&+ r\delta_{r,-s}(\delta_{j,k}\delta_{i,l}-\delta_{i,-k}\delta_{j,-l}) K. 
\end{align*}
$U(\off)$ stands for the universal enveloping algebra of $\off$.
Define $F[r]$ as  $U(\off)$-valued matrix    of size  $2n\times 2n$  with entries $(F[r])_{i,j}=F_{i,j}[r]$ for  $i,j\in{[-n;n]}$.
Then the  matrix   $\tilde F[-1]= F[-1]J_{2n}$  is skew-symmetric  
and, similarly to  the classical case, one defines the Pfaffian of $\tilde F[-1]$ as
\begin{align}\label{defpf}
\text{Pf}\,\tilde F[-1]=\frac{1}{2^n n!}\sum_{\sigma\in S_{[-n;n]} }\text{sgn}\,{(\sigma)}\tilde F_{\sigma(-n),\sigma(-n+1)}[-1]\dots \tilde F_{\sigma(n-1),\sigma(n)}[-1],
\end{align}
where the sum is over all permutations of the elements of the set $[-n; n]$.
By Lemma  4.3 in \cite{M1},    the Pfaffian  of $\tilde F[-1]$ 
is an element of Feigin-Frenkel  center $\mathfrak{z}(\off)$.
Our goal is to prove a combinatorial formula  for $\text{Pf}\,\tilde F[-1]$  that, in particular, allows  easily to compute the image of this Segal-Sugawara vector  under  the Harish-Chandra homomorphism. 

As in \cite{IU}, it is convenient for us to express Pfaffian-like elements as coefficients of powers of certain two-forms. 
Consider the algebra $\Lambda$  generated by elements $\{e_i\}_{i\in[-n;n]}$ with the defining relations  $e_ie_j=-e_je_i$ for all $i$ and $j$. Define  the $U(\off)$-valued  two-form
 \begin{align*}
\Omega[s] =\sum_{i,j\in[-n; n]} e_i e_{-j} F_{i,j}[s]\in \Lambda \otimes U(\off).
 \end{align*}
\begin{lemma}\label{omn}
 \begin{align*}
(\Omega[-1])^n=e_{-n}e_{-n+1}\dots e_{-1}e_{1}\dots e_{n} 2^n n!\text{Pf} \tilde F[-1].
 \end{align*}
\end{lemma}
\begin{proof}  (cf. Proposition 1.1 in  \cite{IU} and  Lemma 4.1  in \cite{TH}):
 \begin{align*}
  \Omega[-1] ^n&=
  \sum_{i_1,j_1,\dots, i_n, j_n\in [-n;n]}
  e_{i_1} e_{-j_1}\dots e_{i_n}e_{-j_n}F_{i_1,j_1}[-1]\dots F_{i_n,j_n}[-1]\\
   &=
  \sum_{i_1,j_1,\dots, i_n, j_n\in [-n;n]}
  e_{i_1} e_{j_1}\dots e_{i_n}e_{j_n}\tilde F_{i_1,j_1}[-1]\dots\tilde F_{i_n,j_n}[-1]\\
   &= e_{-n} e_{-n+1}\dots e_{n-1}e_{n}
  \sum_{\sigma\in S_{[-n;n]}}
\text{sgn} (\sigma) \tilde F_{\sigma(-n),\sigma(-n+1)}[-1]\dots\tilde F_{\sigma(n-1),\sigma(n)}[-1]\\
&=e_{-n} e_{-n+1}\dots e_{n-1}e_{n} \,2^n n! \text{Pf}\,\tilde F[-1].
 \end{align*}
 \end{proof}
We also  introduce the following forms:
\begin{align*}
\Theta[r]=\sum_{ i=-n}^{-1}\sum_ {j=1}^{n} {e_i e_{-j}} F_{i,j}[r],&
\quad 
\tilde \Theta [r]=\sum_ {i=1}^{n} \sum_{ j=-n}^{-1} {e_i e_{-j}} F_{i,j}[r],\\
\Xi[r]=\sum_{i,j=-n}^{-1} {e_i e_{-j}} F_{i,j}[r] ,&\quad \Psi=\sum_{j=1}^{n} e_{j}e_{-j},\\
\xi[-1]=\Xi[-1] +2\Psi\tau, &\quad \tilde \xi[-1]=\Xi[-1] -2\Psi\tau,
\\
X[-1] = \Theta[-1]+\xi[-1],&\quad 
Y[-1] =\tilde \Theta[-1] +\tilde \xi[-1].
\end{align*}

The following properties of the forms hold: 
\begin{lemma}
\begin{align}
 \notag
\Psi=\sum_{j=-n}^{-1} e_{-j}e_{j}&=-\sum_{j=-n}^{-1} e_{j}e_{-j}=-\sum_{j=1}^{n} e_{-j}e_{j}. \\
[\Xi[r],\Xi[s]]&=-r\delta_{r,-s} \Psi^2 K,\\
[\Theta[r], \Theta[s]]&=[\tilde\Theta[r], \tilde\Theta[s]]=0,\label{tha}\\
[\Theta[r], \tilde\Theta[s]]&=-4\Psi \Xi[r+s] +2r\delta_{r,-s} \Psi^2 K,\label{4t}\\ 
[\Xi[r],\Theta[s] ]&= 2 \Psi\Theta[r+s],\quad  [\Xi[r],\tilde\Theta [s]]=-2 \Psi \tilde\Theta[r+s],\label{t1}\\
[\Xi[-1]+2a\Psi\tau ,\Theta[s] ]&= 2 \Psi(1-as)\Theta[s-1], \label{t2}\\ 
[\Xi[-1]+2b\Psi\tau ,\tilde\Theta[s] ]&= 2 \Psi(-1-bs)\tilde \Theta[s-1], \label{t3}\\ 
[X[-1],Y[-1]] &= 0.\label{xy}
\end{align}
\end{lemma}
\begin{proof}
All of  these relations can be checked directly. We illustrate this with  calculations of (\ref{4t}). With $i,l\in\{-n,\dots, -1\}$ and $j,k\in\{1,\dots, n\}$,
\begin{align*}
[\Theta[r], \tilde\Theta[s]]&=
\sum_{ijkl}e_{i}e_{-j} e_k e_{-l}[F_{i,j}[r], F_{k,l}[s]]
=\sum_{ijkl}e_{i}e_{-j} e_k e_{-l}\\
&\times \big(
 \delta_{j,k}F_{i,l}[r+s]+\delta_{i,l}F_{-j,-k}[r+s]
-\delta_{j,-l}F_{i,-k}[r+s] -\delta_{i,-k}F_{-j,l}[r+s]\\
&+r\delta_{r,-s}(\delta_{j,k}\delta_{i,l}-\delta_{i,-k}\delta_{j,-l})K\big)
\\
&= -\Psi\big(
\sum_{i,l}e_{i}e_{-l}F_{i,l}[r+s]
+\sum_{j,k}e_{-j}e_{k}F_{-j,-k}[r+s]\\
&+\sum_{i,k}e_{i}e_{k}F_{i,-k}[r+s]
+\sum_{j,l}e_{-j}e_{-l}F_{-j,l}[r+s]
\big) + 2r\delta_{r,-s} \Psi^2 K
\\
&=-4\Psi \Xi[r+s] +2r\delta_{r,-s} \Psi^2 K.
\end{align*}

\end{proof}

Observe that 
$\Omega[-1] =\tilde \Theta [-1]+2\Xi[-1]+ \tilde \Theta [-1]
=X[-1] + Y[-1]$,
 and since $X[-1]$ and $Y[-1]$ commute, we can write 
\begin{align*}
\Omega[-1]^n &= (X{[-1]}  +Y{[-1]})^n=\sum_{m+k=n}{n\choose m} Y[-1]^mX[-1]^{k}.
\end{align*}
From now on we will use the following notation: for a non-negative integer $a$ and a partition $\alpha\vdash a$, $\alpha=(a_1\ge\dots\ge a_l\ge 0)$, denote as 
$\tilde \Theta[-\alpha] =\tilde \Theta[-a_1]\dots \tilde \Theta[-a_l]$, and  as
$ \Theta[-\alpha] = \Theta[-a_1]\dots \Theta[-a_l] $.
We also set
$\tilde \Theta[-(0)] = \Theta[-(0)]=1$.
\begin{proposition}\label{yx}
\begin{align*}
Y[-1]^m=\sum_{a=0}^{m} 
\sum_{\alpha\vdash  a} \frac{m!}{(m-a)!m_1!m_2!\dots}
 (-2\Psi)^{ a-l(\alpha)}  
\tilde \Theta[-\alpha]\, (\tilde \xi[-1])^{m- a},
\\
X[-1]^m=\sum_{a=0}^{m} 
\sum_{\alpha\vdash  a} \frac{m!}{(m-a)!m_1!m_2!\dots}
 (-2\Psi)^{ a-l(\alpha)}  
( \xi[-1])^{m- a} \Theta[-\alpha] ,
\end{align*}
 where the sum is taken over all partitions $\alpha=(1^{m_1}, 2^{m_2}\dots )\vdash a$, 
 and  $l(\alpha)$
 is the length of $\alpha$:
$ l(\alpha)= m_1+m_2+\dots$\,.
\end{proposition}

\begin{proof}
We prove the statement in two steps. First, we show that $Y[-1]^m$ and $X[-1]^m$  can be expressed as  linear combinations of ordered terms  labeled by {\it compositions} of integer numbers $a$, $1\le a\le m$.
\begin{lemma}\label{compos_lemma}
\begin{align*}
Y[-1]^m=\sum_{a=1}^{m}\sum_{\bar a} C(\bar a,m) (-2\Psi)^{ a-l}  
\tilde \Theta[-a_1] \tilde \Theta[-a_2]  \dots \tilde \Theta[-a_l] \, (\tilde \xi[-1])^{m- a},\\
X[-1]^m=\sum_{a=1}^{m}\sum_{\bar a} C(\bar a,m) (-2\Psi)^{ a-l}  
( \xi[-1])^{m- a}
 \Theta[-a_1] \ \Theta[-a_2]  \dots  \Theta[-a_l] ,
\end{align*}
 where the sum is taken over all ordered $l$-tuples $\bar a=(a_1,\dots, a_l)$,  with $0\le l \le m$, $a_i\ge 1$ 
and $\sum_{i=1}^l a_i=a\le m$. Here
 \begin{align*}
 C(\bar a,m)=
 \frac{m!}{(m-a)!}\prod_{s=1}^{l}\frac{a_s}{(a_1+\dots +a_s)}.
  \end{align*}
\end{lemma}

\begin{proof} 

$
Y[-1]^m=( \tilde \Theta[-1] +\tilde \xi[-1])^m
$
is the sum of monomials of the form
 \begin{align}\label{mon1}
  \tilde \xi[-1]^{p_1} \tilde \Theta[-1]\tilde \xi[-1]^{p_2}  \tilde \Theta[-1]\dots  \tilde \xi[-1]^{p_l} \tilde \Theta[-1]\tilde \xi[-1]^{p_{l+1}} ,
  \end{align}
where $p_i\ge 0$ and 
$
p_1+\dots+ p_{l+1}=m-l.
$
From (\ref{t3}), 
\begin{align}\label{comb}
 \tilde \xi[-1]^{p}  \tilde \Theta[-1]=\sum_{s=1}^{p+1}{p\choose {s-1}} (-2\Psi)^{s-1}  s! \tilde \Theta[-s] \tilde \xi[-1]^{p-s+1},
 \end{align}
 and we obtain  that   the terms of the  monomials   (\ref{mon1}) can be permuted to obtain that  $Y[-1]^m$ as a linear combination 
\begin{align}\label{mon2}
Y[-1]^m=
\sum_{\bar a} c(\bar a, m)(-2\Psi)^{ a-l} a_1! a_2!\dots a_l!
\tilde \Theta[-a_1] \tilde \Theta[-a_2]  \dots \tilde \Theta[-a_l] \, (\tilde \xi[-1])^{m- a}
\end{align}
with certain coefficients  $c(\bar a, m)$  and the sum   taken over all ordered  $l$-tuples $\bar a=(a_1,\dots, a_l)$  that satisfy $a_i\ge 1$\, $(i=1,\dots, l)$,  $a=\sum_{i=1}^l  a_i$. Our goal is  to compute $c(\bar a, m)$.
Fix the ordered  $l$-tuple  $(a_1,\dots, a_l) $ and consider the corresponding term  in the sum (\ref{mon2}). 
We use the following notations:
\begin{align*}
&A_1=a_1-1, \quad A_2=(a_1-1)+(a_2-1),\quad \dots,\quad   A_l=(a_1-1)+\dots + (a_l-1),\\
&P_1=p_1+\dots +p_{l+1}=m-l,\quad  P_2=p_2+\dots +p_{l+1},\quad \dots, \\
&P_{l+1}=p_{l+1},\quad P_{k}=0 \quad  (k>l+1).
\end{align*}
Then from (\ref{comb}),
\begin{align*}
c(\bar a, m)=&\sum_{p_i\ge 0, \, p_1+\dots + p_{l+1}=m-l} {p_1\choose {a_1-1}} {{p_1+p_2-A_1}\choose {a_2-1}} \dots {{\sum_{i=1}^l p_i- A_{l-1}}\choose {a_l-1}}\\
=&\sum_{p_i\ge 0, \, p_1+\dots + p_{l+1}=m-l} {{m-l -P_2}\choose {a_1-1}} {{ m-l -P_3 -A_1}\choose {a_2-1}} \dots {{m-l- P_{l+1}- A_{l-1}}\choose {a_l-1}}\\
=&
\sum_{p_{l+1}=0}^{m-l} 
{{m-l- P_{l+1}- A_{l-1}}\choose {a_l-1}}
\dots 
\sum_{p_{3}=0}^{m-l-P_{4}}
{{ m-l -P_3-A_1}\choose {a_2-1}}  \sum_{p_{2}=0}^{m-l-P_3}
 {{m-l -P_2}\choose {a_1-1}}.
 \end{align*}
We compute the sums of this expression  starting from the last one. Recall that
 $$\sum_{j=0}^{r}{j\choose k}={{r+1}\choose {k+1}} \quad {\text {and}} \quad 
 {{j-b}\choose a}{j\choose b}={j\choose {a+b}}{{a+b}\choose a}.
 $$
We get 
 $$
 \sum_{p_{2}=0}^{m-l-P_3}
 {{m-l -P_2}\choose {a_1-1}}={{m-l+1-P_3}\choose {a_1}}
 $$
and
 $$
  { {m-l -P_3- A_1} \choose {a_2-1}}  {{m-l+1-P_3}\choose {a_1}}
= {{m-l+1-P_3}\choose {a_1+a_2-1}} {{ a_1+a_2-1}\choose {a_2-1}}.
  $$
Similarly,
$$
\sum_{p_s=0}^{m-l-P_{s+1}}  {{m-l+s-2-P_s}\choose {a_1+\dots +a_{s-1}-1}}= {{m-l+s-1-P_{s+1}}\choose {a_1+\dots +a_{s-1}}}
$$
and 
\begin{align*}
{{m-l-P_{s+1}-A_{s-1}}\choose {a_s-1}}{{m-l+s-1-P_{s+1}}\choose {a_1+\dots +a_{s-1}}}={{m-l+s-1-P_{s+1}}\choose {a_1+\dots +a_{s}-1}}
{{a_1+\dots +a_s-1}\choose{a_s-1}}.
\end{align*}
By induction, 
\begin{align*}
c(\bar a, m)=& {{m}\choose {a_1+\dots +a_l}}{{a_1+\dots +a_l-1}\choose{a_l-1}}\dots {{ a_1+a_2-1}\choose {a_2-1}}\\
&=\frac{m!}{(m-a)!}\prod_{s=1}^{l}\frac{1}{(a_1+\dots +a_s) (a_s-1)!}.
\end{align*}
Set  $C(\bar a, m)=a_1! a_2!\dots a_l!\,c(\bar a, m)$, and 
 the   statement of Lemma \ref{compos_lemma} follows from (\ref{mon2}).
 
For $X[-1]^m$ the same arguments apply since
\begin{align*}
 \Theta[-1] \xi[-1]^{p}  =\sum_{s=1}^{p+1}{p\choose {s-1}} (-2\Psi)^{s-1}  s! \xi[-1]^{p-s+1}
 \Theta[-s].
 \end{align*}
\end{proof}
 Observe now  that  $\tilde \Theta[-a]$ and  $\tilde \Theta[-b]$ commute for $a,b\ge 1$.
 Thus the   factors of  any  monomial $\tilde \Theta[-a_1] \tilde \Theta[-a_2]  \dots \tilde \Theta[-a_l]$
 can be rearranged into a monomial  $\tilde \Theta[-a_{i_1}] \tilde \Theta[-a_{i_2}]  \dots \tilde \Theta[-a_{i_l}]$ so that  $(a_{i_1}\ge\dots\ge  a_{i_l}\ge1)$ is a {\it partition}  of $a$, and 
\begin{align*}
Y[-1]^m&=\sum_{a=1}^{m}\sum_{\alpha\vdash  a}  \frac{m!}{(m-a)!}C(\alpha)
\prod_{s=1}^{l}   a_s \,   (-2\Psi)^{ a-l}  
\tilde \Theta[-\alpha] \, (\tilde \xi[-1])^{m- a},
\end{align*}
where the  sum is taken over all  partitions $\alpha=(a_1\ge a_2\ge \dots \ge  a_l\ge 1)$ of $a$, and 
$$
 C(\alpha)=\sum  \frac {1 }  {(a_{i_1}) (a_{i_1}+a_{i_2})\dots ( k_1 a_1\dots +k_la_{l})},
$$
--   the sum is taken over all  distinct permutations $(a_{i_1},\dots, a_{i_l}) $    of the  parts  of $\alpha$.
 
 The following lemma allows to compute the coefficients  of that linear combination. 
 \begin{lemma}\label{calpha}
 Let $\{a_1, a_2,\dots,  a_r)$ be a set of distinct positive numbers and let\\ $\alpha=(a_1,\dots a_1,  a_2,\dots,a_2,\,  \dots \, a_r,\dots, a_r)$
 be an ordered $l$-tuple, where the part $a_i$ appears with multiplicity $k_i$
(therefore,  $l= k_1+\dots + k_r$).
Then 
\begin{align}\label{C1}
\sum \frac{1}{(a_{i_1}) (a_{i_1}+a_{i_2})\dots ( k_1 a_1+\dots +k_ra_{r})} =\frac{1}{k_1! a_1^{k_1}\dots k_r!a_r^{k_r}},
\end{align}
where the sum on the left-hand side is taken over  distinct permutations $(a_{i_1}, \dots, a_{i_{l}})$  of $\alpha=(a_1,\dots a_1,  a_2,\dots,a_2,\,  \dots \, a_r,\dots, a_r)$.
\end{lemma}

\begin{proof}
  Denote the sum on the left-hand side of (\ref{C1}) as $C(\alpha)$.
  Assume first that 
  $k_i=1$ ($i=1,\dots, r$).
Then  for $r=2$,
\begin{align}\label{int}
C((a_1, a_2))=
\frac{1}{a_1(a_1+a_2)}+\frac{1}{a_2(a_1+a_2)}= \frac{1}{a_1a_2}.
\end{align}
 By induction for general $\alpha$ with all  distinct parts we show that 
\begin{align}\label{C}
C(\alpha)&=\sum_{\sigma\in S_r } \prod_{s=1}^r\frac{1}{(a_{\sigma(1)}+\dots +a_{\sigma(s)})}\\
& =
\sum_{k=1}^r\,  \sum_{\sigma\in S_r, \sigma(r)=k }  \frac{1}{(a_{\sigma(1)}+\dots +a_{\sigma(r-1)}+a_k)} \prod_{s=1}^{r-1}\frac{1}{(a_{\sigma(1)}+\dots +a_{\sigma(s-1)})} \\
&=\sum_{k=1}^r \,\sum_{\sigma\in S_r, \sigma(r)=k }  \frac{1}{(a_{\sigma(1)}+\dots +a_{\sigma(r-1)}+a_k)}\frac{1}{a_1 \dots \hat{a}_k\dots a_r }
=\frac{1}{a_1 \dots a_r }.
\end{align}
($\hat{a}_k$ means that the factor ${a}_k$  is omitted).
Now  suppose that not all of $k_i=1$ in $\alpha$. For simplicity of the argument, let us assume that   ${k_1}\ne 1$, and the rest of  $k_i= 1$.
Then consider $\alpha_\varepsilon= (a_1+\varepsilon_1, a_1+\varepsilon_2,\dots, a_{1}+\varepsilon_{k_1}, a_{2},\dots, a_r)$ with such values of $\varepsilon_i$ that all the terms of $\alpha_\varepsilon$ are distinct.  $C(\alpha_\epsilon)$ is a rational function of $(\varepsilon_1,\dots \varepsilon_{k_1})$. When all $\varepsilon_i\to 0 $,
$$
\frac{1} {(a_1+\varepsilon_1)\dots (a_{k_1}+\varepsilon_{k_1}) a_{k_1+1}\dots a_r}\to \frac{1}{(a_1)^{k_1} a_{k_1+1}\dots a_r}.
$$
On the other hand, 
$$
C(\alpha_\epsilon)
\to \sum_{\sigma\in S_{r-k_1+1} } \prod_{s=1}^r\frac{k_1!}{(a_{\sigma(1)}+\dots +a_{\sigma(s)})},
$$
where the sum is taken over all permutations of the set $\{1,2,\dots, r\} $. 
By generalizing this argument one can show that if $\alpha$   has distinct parts $\{ a_1, \dots, a_r\}$ with the corresponding multiplicities
 $\{ k_1,\dots, k_r \}$,
the value of  $C(\alpha)$ is given by the formula
$
C(\alpha)={1}/({k_1! a_1^{k_1}\dots k_r!a_r^{k_r}}).
$
\end{proof}
Applying Lemma \ref{calpha} and  rewriting the partition $\alpha$ in the form  $\alpha=( 1^{m_1}, 2^{m_2},\dots)$, we get the statement of Proposition \ref{yx}.
\end{proof}
\section{Formula for  $\Omega[-1]^n$}
We use notation  $\text{ad}_\tau$ for the operator $f\mapsto [\tau,f]$  acting on  $U(\off)$.
\begin{theorem}\label{thm1} 
For  the Lie algebra $\of$,
\begin{align*}
\Omega[-1]^n &= \sum_{l=0}^{\lfloor n/2\rfloor} \sum_{\substack{a,b,\\ 2l\le a+b\le n}}\frac{2^{n-2l} n!}{(n-a-b)!} 
 \left(\sum_{\alpha\vdash a, l(\alpha)=l}\frac{1}{m_1! m_2!\dots }
  \tilde  \Theta[-\alpha]   \right)\\
&\quad \times(-\Psi)^{ a+b-2l} \left( (\Xi[-1] -\Psi \text{ad}_\tau)^{n-a-b}\cdot 1 \right)\,
\left(\sum_{\beta\vdash b,l(\beta)=l}
\frac{1}{m^\prime_1!m^\prime_2!\dots }
  \Theta[-\beta]\right),
\end{align*}
where  we use the notations $\alpha=(1^{m_1}, 2^{m_2},\dots )\vdash a$ and  $\beta=(1^{m^\prime_1}, 2^{m^\prime_2},\dots )\vdash b$ .

\end{theorem}
\begin{example}
For $n=2$ the partitions that contribute to the sum of the formula are $\alpha=\beta=\emptyset$ and $\alpha=\beta=(1)$. Accordingly, in ${\mathfrak o}_{4}$ we get
\begin{align*}
\Omega[-1]^2=4(\Xi[-1]-\Psi\text{ad}_\tau)^2\cdot 1 + 2 \tilde  \Theta[-1]  \Theta[-1].  
\end{align*}
For $n=3$ the partitions that contribute to the sum of the formula are $\alpha=\beta=\emptyset$; $\alpha=\beta=(1)$;   $\alpha=(2), \beta=(1)$ or  $\alpha=(1), \beta=(2)$. Accordingly, in ${\mathfrak o}_{6}$ we get
\begin{align*}
\Omega[-1]^3&=8(\Xi[-1]-\Psi \text{ad}_\tau)^3\cdot 1 + 12 \tilde  \Theta[-1] \,((\Xi[-1]-\Psi\text{ad}_\tau)\cdot 1)\,  \Theta[-1]\\
 & \quad
-12 \Psi( \tilde  \Theta[-2]   \Theta[-1] +\tilde  \Theta[-1]   \Theta[-2] ).
\end{align*}

\end{example}
\begin{proof}
\begin{align*}
\Omega[-1]^n &= \sum_{m=0}^n{n\choose m} Y[-1]^mX[-1]^{n-m}\\
&=\sum_{m=0}^n \sum_{a=0}^{m}\sum_{b=0}^{n-m} \sum_{\alpha\vdash a\beta\vdash b}  {n\choose m}\frac{m!\,(n-m)!}{ (m-a)!(n-m-b)! m_1! m_2!\dots m^\prime_1!m^\prime_2!\dots }\\
&\quad  (-2\Psi)^{ b+a-l(\alpha)-l(\beta)}  
\tilde \Theta[-\alpha] (\tilde \xi[-1])^{m- a}( \xi[-1])^{n-m- b}\Theta[-\beta]\\
&=\sum_{0\le a+b\le n}\sum_{\alpha, \beta}
\frac{n!\,}{ (n-a-b)!m_1! m_2!\dots m^\prime_1!m^\prime_2!\dots }(-2\Psi)^{ a+b-l(\alpha)-l(\beta)}\\
& \,\times \tilde \Theta[-\alpha]  \left(\sum_{m=a}^{n-b}  {{n-a-b}\choose {m-a}}(\tilde \xi[-1])^{m-a}(\xi[-1])^{n-m- b}\right)\Theta[-\beta].
\end{align*}
Denote as 
$$
P_r=\sum_{k=0}^{r}  {{r}\choose {k}}(\tilde \xi[-1])^{k}(\xi[-1])^{r-k}, \quad P_0=1.
$$
One has:
\begin{align*}
P_r&= \tilde \xi [-1]P_{r-1}+P_{r-1}\xi[-1]
=(\Xi[-1] -2\Psi\tau)P_{r-1}+P_{r-1}(\Xi[-1] +2\Psi\tau)\\
&=\Xi[-1] P_{r-1}+P_{r-1}\Xi[-1] -2\Psi[\tau, P_{r-1}].
\end{align*}
If $P_{r-1}$ is a sum of  monomials in $\Psi$ and  $\Xi[s] $ with some negative values of  $s$, then $[\tau, P_{r-1}]$ is also a sum of monomials in $\Psi$ and  $\Xi[s]$ with some negative values of $s$. This,  together with the  fact that $P_1=2\Xi[-1]$,
implies that all $P_r$ are linear combinations  of monomials in $\Psi$ and  $\Xi[s]$ ($s<0$),  each of them  commutes with $\Xi[-1] $ and  we can write
\begin{align*}
P_r=2(\Xi[-1] P_{r-1}-\Psi[\tau, P_{r-1}])= 2(\Xi[-1] -\Psi\text{ad}_\tau)\, P_{r-1}=2^r(\Xi[-1] -\Psi\text{ad}_\tau)^r\cdot 1.
\end{align*}
Thus
\begin{align*}
\left(\sum_{m=a}^{n-b}  {{n-a-b}\choose {m-a}}(\tilde \xi[-1])^{m-a}(\xi[-1])^{n-m- b}\right)=
   \left(   2^{n-a-b}(\Xi[-1] -\Psi\text{ad}_\tau)^{n-a-b}\cdot 1 \right).
   \end{align*}
   Also  notice that since $\Omega[-1]^n$ is a  form of full degree,  the number of parts $\alpha$ and $\beta$ in non-zero terms is the same: $l(\alpha)=l(\beta)=l$. 
   This completes the proof of  the theorem.
\end{proof}

\section{Explicit formula for Pfaffian  $\Pf  \tilde F[-1]$}

  Pfaffian  $\text{Pf}\,\tilde F[-1]$ is given by  the coefficient of the monomial  $e_{-n}e_{-n+1}\dots e_{n-1}e_{n}$
  in the  form  $\Omega[-1]^n$ . 
 One can use Lemma \ref{pfdet} below to interpret through Pfaffian- and determinant-like elements the  expressions $\tilde \Theta[-\alpha]$,  $(\Xi[-1] -\Psi \text{ad}_\tau)^{r}$, $\Theta[-\beta]$
in the statement of  Theorem \ref{thm1}. 
 
 For  subsets $I, J$ of $[-n;n]$ such that $|I|=|J|=l$ we denote as  $\Phi[-s]_{IJ}$  an $l\times l$-submatrix of the matrix $\Phi[-s]=(F[-s]+{\text{ad}_\tau}\cdot Id )$
\begin{align*}
\Phi_{IJ} [-s]= (F_{i,j}[-s]+ {\text{ad}_\tau} \delta_{ij})_{i\in I, j\in J}.
\end{align*}
and set 
\begin{align*}
 \text{det} \, \Phi_{IJ} [-1]=\sum_{\sigma\in S_l}\text{sgn}\,(\sigma) \Phi_{i_{\sigma (1)}, j_1}\dots \Phi_{i_{\sigma(l)}, j_l}.
\end{align*}

Also we define  for  subsets $I=( i_1<\dots <i_l)$    of $\{-n,\dots, -1\}$ the elements
\begin{align}
\text{Pf}\, \Phi_{+I-I} [-\beta]&=\frac{1}{2^l l!}\sum_{\sigma\in S_{2l}} \text{sgn}\,(\sigma) F_{i_{\sigma(1)}, -i_{\sigma (2)}}[-\beta_1]\dots  F_{i_{\sigma(2l-1)}, -i_{\sigma (2l)}}[-\beta_l],\label{pf2}\\
\text{Pf}\, \Phi_{-I+I} [-\alpha]&=\frac{1}{2^l l!}\sum_{\sigma\in S_{2l}} \text{sgn}\,(\sigma) F_{-i_{\sigma(1)}, i_{\sigma (2)}}[-\alpha_1]\dots  F_{-i_{\sigma(2l-1)}, i_{\sigma (2l)}}[-\alpha_l]. \label{pf3}
\end{align}
\begin{lemma}\label{pfdet}
For  partitions $\alpha$, $\beta$ and for   $r\in \bZ_+$,
\begin{align*}
\tilde \Theta[-\alpha] &=2^l l!\sum_{-n\le i_1<\dots <i_{2l}\le -1} e_{-i_1}\dots  e_{-i_{2l}}\,  \text{Pf}\, \Phi_{-I+I} [-\alpha],\\
 \Theta[-\beta] &=2^l l!\sum_{-n\le i_1<\dots <i_{2l}\le -1} e_{i_1}\dots  e_{i_{2l}} \, \text{Pf}\,  \Phi_{+I-I} [-\beta],\\
 (\Xi[-1] -\Psi \text{ad}_\tau)^{r}\cdot 1&= r! \sum_ {\substack{
  -n\le i_1<\dots<i_r\le -1\\
   -n\le j_1<\dots<j_r\le -1}}
 e_{i_1}\dots e_{i_r}e_{-j_r}\dots e_{-j_1}\, (\text{det}\, \Phi_{I,J} [-1])\cdot1. 
\end{align*}

 \end{lemma}
\begin{proof}
The first and the second equalities follow from the definitions (\ref{pf2}-\ref{pf3}), and the proof  of the third equality  repeats  word-to-word the proof of  Proposition 4.5.  in \cite{TH} if we observe that
$\Xi[-1]-\Psi \text{ad}_\tau=\sum_{j=-n}^{-1}\eta_j e_{-j}$, where $\eta_{j}=\sum_{i=-n}^{-1}e_i(F_{i,j}[-1]+\text{ad}_\tau\delta_{ij})$ and satisfy  $\eta_j\eta_k+\eta_k \eta_j=0$.
\end{proof}
\begin{example}\label{ex_2}
 For $n=2$ one gets 
$
4(\Xi[-1]-\Psi \text{ad}_\tau)^2\,=\, 8 e_{-2}e_{-1} e_{1}e_{2}\, \text{det}(\Phi[-1])\cdot 1
$
and $\Theta[-1]=2e_{-2}e_{-1}F_{-2,1}[-1]$,
 $\tilde \Theta[-1]=2e_{1}e_{2}F_{1,-2}[-1]$. Therefore, 
\begin{align*}
& \Pf  \tilde F[-1]\,=\,\text{det}(\Phi[-1])\cdot 1-F_{1,-2}[-1]F_{-2,1}[-1]\\
&=\, (F_{-2,-2}[-1]F_{-1,-1}[-1] - F_{-1,-2}[-1]F_{-2,-1}[-1]  + F_{1,-2}[-1]F_{-2,1}[-1] + F_{-1,-1}[-2]).
\end{align*}
\end{example}
This example illustrates that Theorem \ref{thm1}  can be used  to compute the Pfaffian $\Pf  \tilde F[-1]$:   we find  by  that theorem the value of $\Omega[-1]^n$, and  then by Lemma \ref{omn},  the coefficient of the monomial  $e_{-n} e_{-n+1}\dots e_{n-1}e_{n}$  in  $\Omega[-1]^n$ is  $2^n n!\Pf  \tilde F[-1]$. For the sake of completeness we write explicitly  the general  formula for that coefficient (c.f. Theorem 4.7 in \cite {TH}), though it involves more technical notations.  The statements of Theorems \ref{thm1}  above and \ref{thm2}  below are equivalent.


If  $I=( i_1<\dots <i_k)$  is  a  string   of elements of  subset  of    of $\{-n,\dots, -1\}$ written in increasing order,  denote as  $- I= ( -i_1>\dots >-i_k)$.
Consider subsets  $J, I_1,I_2,J_1,J_2$  of  $\{-n,\dots,-1\}$ that satisfy the properties:
\begin{align}
I_1 \sqcup J_1 \sqcup J= J \sqcup J_2 \sqcup I_2=\{-n,\dots,-1\},\label{subsets1}\\
|I_1|=|I_2|=2l, \quad  |J_1|=|J_2|=n-a-b.\label{subsets2}
\end{align}

Write the elements  of  strings    $[J,J_2, I_2]$   exactly   in that order.   We denote  by 
$ \text{sgn}\, ({ J,J_2, I_2})$  the sign of the  permutation of   the elements of this string that puts  them in the  increasing order  $( -n,\dots, -1 )$.
Similarly,  $\text{sgn}\, ( {-I_1,-J_1,-J})$  is the sign of the  permutation of   the elements of this string that puts  them in the  increasing order  $( 1,\dots, n )$.

 \begin{theorem}\label{thm2}
  \begin{align}
 \Pf  \tilde F[-1]&=\sum_{l=0}^{\lfloor n/2\rfloor}
\sum_{\substack{ a,b\\0\le a+b\le n }}\sum_{I_1,J_1,I_2,J_2, J } 
\text{sgn}\, ( {J,J_2,I_2})\text{sgn}\, ( {-I_1,-J_1,-J})\notag\\
&\quad \times \left( 
\sum_{\alpha\vdash a, l(\alpha)=l}\frac{l!}{m_1! m_2!\dots}
\text{Pf}\, \Phi_{-I_1+I_1} [-\alpha]\right)
\notag\\
&\quad \times 
\left(
{(a+b-2l)! } 
 (\text{det}\, \Phi_{J_2,J_1} [-1]\cdot1)
 \right)
 \left(  \sum_{\beta\vdash b, l(\beta)=l}\frac{l!}{m^\prime_1! m^\prime_2!\dots}\Phi_{+I_2-I_2} [-\beta]\right),\label{uzhas}
 \end{align}
the third sum  is taken over subsets $J, I_1,I_2,J_1,J_2$  of  $\{-n,\dots,-1\}$ that satisfy the properties (\ref{subsets1}, \ref{subsets2}).
 \end{theorem}
\begin{proof} 
Using  for a subset  $I=( i_1<\dots <i_k)$    of $\{-n,\dots, -1\}$ the  notations
$
{e_{I}}=  e_{i_1}\dots  e_{i_{k}}$,
${e_{-I}}= e_{-i_1}\dots  e_{-i_{k}}$, $
{e_{-\tilde I}}=e_{-i_k}\dots  e_{-i_{1}},
$
combine Theorem  \ref{thm1}   and Lemma \ref{pfdet}  to obtain
\begin{align*}
\Omega[-1]^n &= 2^{n} n! \sum_{l=0}^{\lfloor n/2\rfloor}
\sum_{0\le a+b\le n} \sum_{I_1,J_1,I_2,J_2}  
(-\Psi)^{ a+b-2l}
{e_{-I_1}} {e_{J_2}} {e_{-\tilde {J}_1}} {e_{I_2}}\\
&\quad \times \sum_{\alpha\vdash a, l(\alpha)=l}\frac{l!}{m_1! m_2!\dots}
\text{Pf}\, \Phi_ {-I_1 I_1}[-\alpha]  ( \text{det}\, \Phi_ {J_2 J_1}[-1]  \cdot 1 )\\
&\quad \times
\sum_{\beta \vdash b, l(\beta)=l}\frac{l!}{m^\prime_1! m^\prime_2!\dots}
\text{Pf}\, \Phi_ {I_2 -I_2}[-\beta]  
\end{align*}
Note that 
\begin{align*}
(-\Psi)^{ r}&= (r)!\sum_{ J\subset\{-n\dots -1\},\, |J|=r} (-1)^{\frac{r(r-1)}{2}} e_{J} e_{-J},
\end{align*}
and that 
$$ (-1)^{\frac{{(a+b-2l)}(a+b-2l-1)}{2}} e_{J} e_{-J}{e_{-I_1}} {e_{J_2}} {e_{-\tilde {J}_1}} {e_{I_2}}
= {e_{J}} {e_{{J}_2}} {e_{I_2}}e_{-I_1} e_{-J_1} {e_{-J}}.$$

By comparing the  order of  the  terms of this monomial    with the monomial  $e_{-n}e_{-n+1}\dots\\ e_{-1}e_1\dots e_{n-1}e_{n}$ we get  formula (\ref{uzhas}).
\end{proof}

\section{Harish-Chandra image of the Pfaffian }\label{sec_HC}
The  Feigin-Frenkel center $\mathfrak{z}(\off)$ is a commutative subalgebra of $U(t^{-1}\of[t^{-1}]) ^\fh$. Consider  a left ideal  $ I$    of $U(t^{-1}\of[t^{-1}])$ that is  generated by the elements $F_{ij}[r]$ for $-n\le i<j\le n$ and $r<0$. The quotient of   $U(t^{-1}\of[t^{-1}]) ^\fh$ by the  two-sided ideal  $U(t^{-1}\of[t^{-1}]) ^\fh\cap  I$ is  isomorphic to  the commutative algebra  that is freely generated  by the images of the elements 
 $F_{ii}[r]$ ($i=-n,\dots, -1$,  $r<0$). Denote these images as $\mu_i[r]=F_{ii}[r]$.
 We get an analogue of  classical Harish-Chandra homomorphism 
$$\chi: U(t^{-1}\of[t^{-1}]) ^\fh\to U(t^{-1}\fh[t^{-1}]), $$
and  $U(t^{-1}\fh[t^{-1}]) ^\fh$ can be viewed as a polynomial algebra in variables $\{\mu_i[r], i=-n, \dots,-1, r<0\}$.
The  restriction  of $\chi$ to  $\mathfrak{z}(\off)$ yields an isomorphism
 $\mathfrak{z}(\off)\to \mathcal W(\of)$, where $\mathcal W(\of)$ is the classical $\mathcal{W}$-algebra of $\of$. It is a subalgebra of $U(t^{-1}\of[t^{-1}]) ^\fh$ annihilated by screening operators (\cite{F}, Chapter 7).
The elements of $\mathcal W(\of)$ are polynomials in $\mu_i[r]$ ($r<0$, $i=-n,\dots, -1$). 
We give a new direct proof of the following corollary established in \cite{MM}.
\begin{corollary}\label{HC}
\begin{align*}
\chi( \Pf \tilde F[-1])= (\mu_{-n}[-1] + \text{ad}_\tau)\dots(\mu_{-1}[-1] +\text{ad}_\tau)\cdot1.
\end{align*}
\end{corollary}
\begin{proof}
 Only  the terms with $a=b=0$ will survive under Harish-Chandra homomorphism in the formula  of Theorem \ref{thm1} for $\Omega[-1]^n$:   
\begin{align*}
\chi(\Omega [-1]^n)&=\chi (
2^n (\Xi[-1] -T\text{ad}_\tau)^{n}\cdot 1)=
2^nn!\, e_{1}\dots e_n e_{-1}\dots e_{-n}\,   \chi (\text{det}\,(F [-1]+\text{ad}_\tau)\cdot1)\\
&=2^n n! \,  e_{1}\dots e_n e_{-1}\dots e_{-n}\, (\mu_{-n}[-1]+\text{ad}_\tau )\dots(\mu_{-1}[-1]+\text{ad}_\tau )\cdot 1
\end{align*}
\end{proof}

\section{Deducing  the  formula for  Pfaffian for  $U(\of)$}\label{fol}
In this section we would like to illustrate that the application of evaluation homomorphism to the statement of Theorem \ref{thm1}  implies the formulas for Pfaffian of the universal enveloping algebra $U(\of)$  in \cite{TH} (namely, Proposition 4.4.  there).

 Let $z$ be another variable and consider a homomorphism of algebras  $\varphi:\,  U(t^{-1}\of[t^{-1}])\to   U(\of)\otimes z^{-1} \bC\langle z^{-1} \rangle $, defined  by  $\varphi(F_{ij}[r])= F_{ij}\otimes z^r$,  $\varphi(K)= 0$.
  Then the action of $\text{ad}_\tau$ is identified with the operator $-\partial_z$ by the rule
  $\varphi([\tau, f])=-\partial_z(\varphi(f))$. 
Let us apply $\varphi$ to the formula of the Theorem \ref{thm1}.

We have:
\begin{align*}
\varphi(\Omega [-1]^n)=\Omega ^n z^{-n}, \quad \text{where}\quad  \Omega=\sum_{i,j\in[-n;n]}e_ie_{-j} F_{i,j}\in \Lambda\otimes U(\of).
\end{align*}
Note that  for  a  partition $\alpha\vdash a$ of length $l(\alpha)=l$, the  evaluation map gives
$\varphi(\tilde \Theta [-\alpha])= \tilde \Theta^l z^{-a}\quad \text{and}\quad \varphi(\Theta [-\alpha])= \Theta^l z^{-a},$ where $ \tilde \Theta=\sum_{i=1}^{n}\sum_{j=-n}^{-1}e_ie_{-j} F_{i,j}$,
$\Theta=\sum_{i=-n}^{1}\sum_{j=1}^{n}e_ie_{-j} F_{i,j}$.
Also 
$$ \varphi ((\Xi[-1] -\Psi\, \text{ad}_\tau)^{k}\cdot 1)=
((\Xi z^{-1} +\Psi\, \partial_z)^{k}\cdot 1), $$
where $
\Xi=\sum_{i,j=-n}^{-1} {e_i e_{-j}} F_{i,j}
$.
Using that  in the notations $\alpha =(1^{m_1}, 2^{m_2},\dots)$, 
\begin{align}
\sum_{\alpha\vdash a, l(\alpha)=l}\frac{l!}{m_1! m_2!\dots}={{a-1} \choose {l-1}}
\end{align}
(see e.g. Corollary 6 in \cite{CW}),
and that
\begin{align}
\sum_{a+b=p}{{a-1} \choose {l-1}}{{b-1} \choose {l-1}}= {{p-1} \choose {2l-1}},
\end{align}
we obtain from Theorem \ref{thm1},
\begin{align}\label {z-p}
\Omega^n z^{-n} &= \sum_{l=0}^{\lfloor n/2\rfloor}  \frac{2^{n-2l} n!}{l!l!} 
  \tilde  \Theta^l \quad\Big( \sum_{2l\le p\le n} \frac{1}{(n-p)!}{{p-1} 
\choose {2l-1}}\\
&\quad \times (-\Psi)^{ p-2l} \,(\Xi z^{-1} +\Psi\partial_z)^{n-p}\cdot 1\,
\Big)  \Theta^l z^{-p}.
\end{align}
Let us multiply both parts of (\ref{z-p})  by $z^{-n}$. 
Observe that 
$$
((\Xi z^{-1} +\Psi\partial_z)^{n-p}\cdot 1) z^{n-p}=  (\Xi)(\Xi-\Psi)(\Xi -2\Psi)\dots (\Xi-(n-p+1)\Psi),
$$
and
\begin{align*}
\frac{1}{(n-p)!}{{p-1} 
\choose {2l-1}}=
\frac{1}{(n-2l)!}{{n-2l}\choose{n-p}}\frac{(p-1)!}{(2l-1)!}.
\end{align*}
Therefore, we can rewrite the  central part expression of  (\ref{z-p})  as
\begin{align}\label{xyz}
&\sum_{2l\le p\le n} \frac{1}{(n-p)!}{{p-1} 
\choose {2l-1}}
 (-\Psi)^{ p-2l} \,((\Xi z^{-1} +\Psi \partial_z)^{n-p}\cdot 1)\, z^{n-p}\notag
\\
=
& \frac{1}{(n-2l)!} \sum_{p=0}^{n}{{n-2l}\choose{n-p}}y^{(p-2l)}x^{(n-p)}
\end{align}
with
\begin{align*}
 y=-2l\Psi,&\quad y^{(p-2l)}=  (-2l \Psi)\dots  (-(p-1)\Psi),\\
x=\Xi, &\quad x^{(n-p)}=\Xi(\Xi-\Psi)\dots (\Xi-(n-p+1)\Psi).
\end{align*}
Since $\Xi$  and $\Psi$ commute, (\ref{xyz}) can be rewritten using binomial theorem for falling factorial powers  as 
\begin{align*}
 \frac{1}{(n-2l)!} (x+y)^{(n-2l)}= \frac{1}{(n-2l)!} (\Xi-2l\Psi)^{(n-2l)}= \frac{1}{(n-2l)!}(\Xi-2l\Psi)\dots (\Xi-(n+1)\Psi).
\end{align*}
Hence, from  (\ref{z-p}),
\begin{align}
\Omega^n=\sum_{l=0}^{\lfloor n/2\rfloor}  \frac{2^{n-2l} n!}{l!l!} 
  \tilde  \Theta^l 
\frac{1}{(n-2l)!}(\Xi-2l\Psi)\dots (\Xi-(-n+1)\Psi)
 \Theta^l .
\end{align}
Finally, using the commotion relation 
\begin{align*}
 \tilde  \Theta^l (\Xi-2l\Psi)\dots (\Xi+(n-1)\Psi)=\Xi (\Xi-\Psi)\dots (\Xi-(n-2l-1)\Psi)\tilde \Theta^l,
\end{align*}
we get 
\begin{align}\label{imp}
\Omega^n=\sum_{l=0}^{\lfloor n/2\rfloor}  \frac{2^{n-2l} n!}{(n-2l)! l!l!} 
\Xi(\Xi-\Psi)\dots (\Xi-(n-2l-1)\Psi)\
 \tilde  \Theta^l  \Theta^l, 
\end{align}
which is the statement of Proposition 4.4. in \cite{TH}.
\begin{remark} (a) As  it is discussed  in the proof of Theorem 4.17 of \cite{TH}, necessarily  $p=q$ in Proposition 4.4 of \cite{TH}.\\
(b) Our notation  $\Psi$ stands for $-\tau$ in \cite{TH}.\\
(c)  As it was mentioned before,  Theorem \ref{thm1} and Theorem \ref{thm2} of this note are equivalent statements. Similarly, the results of \cite{TH} cited here in  (\ref{imp}) and (\ref {HP})
are equivalent.  We chose to illustrate how  (\ref{imp}) is deduced from Theorem \ref{thm1} rather then  (\ref {HP}) from Theorem \ref{thm2} in order to avoid technically involved notations of the other two statements.
\end{remark}


\begin{thebibliography}{99}
 
 \bibitem{F}
 {E. Frenkel,
 {\it Langlands correspondence for loop groups}, Cambridge Studies in Advanced Mathematics, 103, Cambridge University Press,  Cambridge, 2007.
 }
   \bibitem{MC}
 { A.V. Chervov  and A. I. Molev,
 {\it On  higher order Sugawara operators}, Int. Math. Res. Not. (2009),1612--1635. }
 
  \bibitem{CT}
 { A.V. Chervov  and  D. Talalaev,
 {\it Quantum spectral curves, quantum integrable systems and the geometric Langlands correspondence}, arXiv: hep-th/0604128. }
 
 \bibitem{TH}
{T. Hashimoto,
{\it  A central element in the universal enveloping algebra of type $D_n$ via minor summation formula of Pfaffians},
J. Lie Theory 18 (2008), no. 3, 581 --594. }

\bibitem{HU}
 { R., Howe  and  T. Umeda,
 {\it The Capelli identity, the double commatant theorem, and  multiplicity-free actions},
 Math. Ann. 290 (1991), 565--619.
 }


\bibitem{M1}
 {A. I. Molev,
 {\it  Feigin-Frenkel center in types B, C and D}, Invent. math, 191(2013), 1--34. }
 
 
 \bibitem{MM}
 {A. I. Molev and E.E. Mukhin,
 {\it  Yangian characters and classical $\mathcal{W}$-algebras}, arXiv:1212- 4032. }
 
 \bibitem{MN}
 {A. I. Molev and M. Nazarov,
{\it Capelli identities for classical Lie algebras},  Math. Ann. 313 (1999), no. 2, 315 --357. }
 
 \bibitem{IU}
 {M. Itoh and T.Umeda,
 {\it  On central elements in the universal enveloping algebras of the orthogonal Lie algebra},
 Composition Math. 127 (2001), 333-359.
 }
  \bibitem{PP}
{A.M. Perelomov,  and V. S. Popov, 
{\it Casimir operators for the orthogonal and symplectic groups}, 
Soviet J. Nuclear Phys. 3 (1966) 819--824. }
 
\bibitem{CW}
 {C.Wenchang,
 {\it A partition identity on binomial coefficients  and its applications},
 Graphs and Combinatorics 5 (1989), 197-200.
 }

\bibitem{AW}
 {A. Wachi,
 {\it Central elements in the universal enveloping algebras for the split realization of the orthogonal Lie algebras},
 Lett. in  Math.  Phys. 77 (2006), 155--168.
 }

\end{thebibliography}
\end{document}